\documentclass[12pt,twoside]{amsart}
\usepackage{amsmath}
\usepackage{amssymb}
\usepackage{wasysym}
\usepackage{psfrag}
\usepackage{graphicx,epsf,amsmath}  
\usepackage{epsf,graphicx}
\usepackage{comment}

\newtheorem{thm}{Theorem}[section]
\newtheorem{theorem}{Theorem}[section]

\newtheorem{remark}[thm]{Remark}

\newtheorem{lemma}[thm]{Lemma}

\title[]{Isoperimetric inequality for the third eigenvalue of the Laplace-Beltrami operator on $\mathbb S^2$}
\author[n. nadirashvili]{Nikolai Nadirashvili}
\author[Y. Sire]{Yannick SIRE}

\begin{document}
\maketitle
\begin{abstract}
We prove an Hersch's type isoperimetric inequality for the third positive eigenvalue on $\mathbb S^2$. Our method builds on the theory we developed to construct extremal metrics on Riemannian surfaces in conformal classes for any eigenvalue. 
\end{abstract}

\tableofcontents
\section{Introduction}

Let $(\mathbb S^2,g)$ be a Riemannian manifold diffeomorphic to the two-sphere. Denote by 
$$
0=\lambda_0 <\lambda_1\leq \lambda_2\leq ....
$$
the (discrete) spectrum of the Laplace-Beltrami operator on $(\mathbb S^2,g)$. Consider now, the following quantity 
$$
\Lambda_k(\mathbb S^2)=\sup_{g} \lambda_k(g) A_g(\mathbb S^2)
$$
the {\sl extremal} spectrum of $\mathbb S^2$, i.e. the suprema for fixed area (the area of $M$ is denoted $A_g(M)$) of the eigenvalues $\lambda_k$ among {\sl all } metrics on $M$. 

In \cite{hersch}, Hersch proved that 
$$
\Lambda_1(\mathbb S^2) = 8\pi. 
$$

Similarly, in \cite{nadirSphere}, the first author proved that 
$$
\Lambda_2(\mathbb S^2) = 16 \pi . 
$$
Our method to prove the main Theorem of this paper provides actually a simpler proof of this result. 

The goal of the present paper is to prove 
\begin{theorem}\label{lambda3}
The following holds
$$
\Lambda_3(\mathbb S^2) = 24\pi. 
$$
\end{theorem}

Notice that $\lambda_3(g) A_g(\mathbb S^2)$ tends to $24\pi$ when $(\mathbb S^2,g)$ is bubbling into three
equal round spheres.

The proof relies on the theory we developped in \cite{NS1,NS2} for suprema of eigenvalues in conformal classes and minimal submanifolds on the sphere (see Section \ref{conf} for a reminder of the main results), which connects
$\lambda_k$-maximizing metrics on surfaces with harmonic maps of Riemannian surfaces
into $\mathbb S^n$. Notice that already in \cite{Ber} Berger pointed out a connection of certain extremal 
eigenvalue problems with minimal submanifolds of the spheres and in  \cite{YY} Yang and Yau 
proved an isoperimetric inequality for $\lambda_1$ on the projective plane related to the problem with some properties of minimal surfaces in $\mathbb S^4$.

 For  higher eigenvalues on the sphere we make the following conjecture 

\subsection*{Conjecture} The following holds 
$$
\Lambda_k(\mathbb S^2) = 8\pi k. 
$$

\section{Suprema of eigenvalues on conformal classes}\label{conf}

Let $M$ be a compact, boundaryless, connected, smooth Riemannian surface. 
Instead of considering the quantity $\Lambda_k(M)$, we will restrict the supremum to the conformal class of a given background metric $g_{round}$ of the surface $M$. 
We define  
$$\tilde \Lambda_k(M, [g])=\sup_{\tilde g \in [ g ],\,\,A_{\tilde g}(M)=1} \lambda_k(\tilde g). $$
First,  in \cite{NS2} we proved the following result.

\begin{theorem}\label{main1}
Let $(M,g)$ be a smooth connected compact boundaryless Riemannian surface.  For any $k \geq 1$ and a sequence of metrics $g'_{i}=(g'_i)_{n \geq 1}\in [g]$ of the form $ g'_i=\mu'_i g$ such that
$$
\lim_{i\to \infty} \lambda_k(g'_i)=\tilde \Lambda_k(M, [g])
$$
there exists a subsequence of metrics $g'_{i_n}=(g_n)_{n \geq 1}\in [g], \,  g_n=\mu_n g$ such that
$$
\lim_{n\to \infty} \lambda_k(g_n)=\tilde \Lambda_k(M, [g])
$$
 and  a probability measure $\mu$ such that
$$
\mu_{n}\rightharpoonup^*  \mu \,\,\,\,\text{weakly in measure as }  n\to +\infty. 
$$
 Moreover the following decomposition holds
\begin{equation}\label{decomp}
\mu=\mu_r+\mu_s
\end{equation}
where  $\mu_r $ is a $C^\infty$ nonnegative function and $\mu_s$ is the singular part given, if not trivial, by 
$$
\mu_s=\sum_{i=1}^K c_i \delta_{x_i}
$$
for some $K \geq 1$, $c_i  \geq 0$ and some bubbling points $x_i \in M$. Furthermore, the number $K$ satisfies the bound 
$$
K \leq k-1
$$
Moreover, the weights $c_i$ satisfy: there exists $m_j$ such that $1\leq m_j\leq k$ and 
\begin{equation}\label{weightSing}
c_j = \frac{\tilde \Lambda_{m_j}(\mathbb S^2,[g_{round}])}{\tilde \Lambda_k(M, [g])}. 
\end{equation}
The regular part of the limit density $\mu$, i.e. $\mu_r$ is either identically  zero or $\mu_r$ is absolutely continuous with respect to the Riemannian measure with a  smooth positive density vanishing at most at a finite number of points on $M$. 

Furthermore, if we denote $A_r$ the volume of the regular part $\mu_r$, i.e. $A_r=A_{\mu_rg}(M)$, then $A_r$ satisfies: there exists $m_0$ such that $1\leq m_0\leq k$ and 
\begin{equation}\label{weightReg}
A_r = \frac{\tilde \Lambda_{m_0}(M, [g])}{\tilde \Lambda_k(M, [g])} .
\end{equation}

Finally, if we denote $\mathcal U$ the eigenspace of the Laplacian on $(M, \mu_rg)$ associated to the eigenvalue $\tilde \Lambda_k(M, [g])$, then there exists a family of eigenvectors $\left \{ u_1,\cdot \cdot \cdot,u_\ell \right \} \subset \mathcal U$ such that the map  

\begin{equation}
\left \{
\begin{array}{c}
\phi: M\to \mathbb R^\ell\\
x \to (u_1,\cdot \cdot \cdot,u_\ell)
\end{array} \right. 
\end{equation}  
is a harmonic map into the sphere $\mathbb S^{\ell-1}$.

\end{theorem}

Notice that in the case when $M$ is a sphere the map $\phi$ is automatically conformal and hence $\phi$ is a
minimal immersion \cite{SYbook2}. Assuming in this case that $\mathbb S^{\ell-1}\subset\mathbb R^\ell$ is
a unit sphere $M=\mathbb S^2$ and $\tilde g$ is a metric induced on $M$ by the map
$$
\phi: M\to \mathbb R^\ell
$$
we will have $\tilde g = \mu_r g \tilde \Lambda_k(M, [g])/2$

To prove Theorem \ref{lambda3}, we heavily rely on some properties of the minimal immersions from $\mathbb S^2$ into Euclidean spheres, which is a crucial statement in our previous theorem. In the following we denote by $\mathbb D_r(x)$ the disk of radius $r$ and center $x \in M$.

\section{Proof of Theorem \ref{lambda3}}

We divide the proof into several steps. The general idea of the proof is by contradiction on the triviality of the singular part of the limiting measure $\mu_s$ in Theorem \ref{main1}. Indeed the proof goes as follows: in a first step, we prove that if the singular part is non-identically vanishing then Theorem \ref{lambda3} follows. In a second step, we assume the contrary and reach a contradiction. This latter step is the main part of the proof of Theorem \ref{lambda3}.  

\subsection{Bubbling phenomenon and proof of Theorem \ref{lambda3}}

We start by a lemma on the decomposition of the spectrum in the case on a singular extremal metric (i.e. $\mu_s$ is not identically zero). 

\begin{lemma}\label{decomp}
Let $\mu_n \,g$ be the sequence of metrics in Theorem \ref{main1} with eigenvalues $\left \{ \lambda^n_i \right \}_{i \geq 0}$. Consider a smooth cut-off function $\psi$ on $M$ such that $ 0 \leq \psi \leq 1$, $\psi=0$ on $\mathbb D_r(\tilde x)$ and $\psi=1$ on $M \backslash \mathbb D_{2r}(\tilde x)$  where $\tilde x $ is a blow up point in Theorem \ref{main1} . Define the sequence of metrics $h_n=2^{-n} \psi \mu_n$ and $\rho_n$ on $\mathbb S^2$ such that $(\mathbb S^2_+,\rho_n)$ is isometric to $(\mathbb D^2, \mu_n-h_n)$. We extend $\rho_n$ by $0$ on $\mathbb S^2_-$. Denote by $\left \{ \alpha^n_i \right \}_{i \geq 0}$ and $\left \{ \beta^n_i \right \}_{i \geq 0}$ the sequences of eigenvalues of the Laplace-Beltrami operator on $(M,h_n)$ and $(\mathbb S^2,\rho_n)$ respectively. 

Then the following holds: suppose for a natural number $N \geq 1$ the following limit exists:
$$
\lim_{n \to \infty} \lambda^n_i=\lambda_i
$$
for $i=0,...,N$. Then there exists a subsequence $n_m$ and natural numbers $N_1,N_2 \geq 1$ such that the following limits hold 
$$
\lim_{m \to \infty} \alpha^{n_m}_i=\alpha_i
$$
and
$$
\lim_{m \to \infty} \beta^{n_m}_i=\beta_i
$$
and furthermore 
$$
\left \{ \lambda_0,...,\lambda_N \right \}=\left \{ \alpha_0,...,\alpha_{N_1} \right \} \cup \left \{ \beta_0,...,\beta_{N_2} \right \}
$$
where the union of sets is taken considering the multiplicity of the eigenvalues. 
\end{lemma}

\begin{proof}
Consider $\mu_n$ the sequence of densities given by theorem \ref{main1} with eigenvalues $\left \{ \lambda^n_i \right \}_{i \in \mathbb N}$ on $(M,\mu_n g)$. The densities $\mu_n$ are smooth nonnegative functions that might vanish on subsets of $M$. In that case, the eigenvalues are given by the standard variational formulae. Recall that the sequence of densities $\mu_n$ is bounded. Let $\mathbb D^2 \subset M$ be a disk and $x_0 \in \mathbb D^2$. By boundedness of $\mu_n$, we can assume that for any neighborhood $G \subset \mathbb D^2$ of $x_0$ there is $N$ such that for all $n \geq N$, $\mu_n <1$ on $\mathbb D^2 \backslash G$.  Denote by $u_i^n$ the eigenvectors of the Laplace-Beltrami operator asscociated to $\lambda_i^n$ on $(M,\mu_n)$. Let $U^n$ be the linear span of $u_1^n,...,u_N^n$ and consider $u \in U^n$ with the nomalization $\| u\|_{L^2(M)}=1$. Then of course, one has $\| \nabla u\|_{L^2(M)} \leq C$ independently of $n$. Define the sets
$$
\Omega_m =\mathbb D_{2^{-m}}(\tilde x)\backslash \mathbb D_{4^{-m}}(\tilde x). 
$$
Now remember that each function $u^i_n$ satisfies a Schr\"odinger equation with bounded potential. So by standard elliptic estimates (see also \cite{NS1}), one has that there exists $\epsilon \in (0,1)$ such that 
$$
\|u^n_i\|_{L^2(\Omega_m)} \leq C (1-\epsilon)^m
$$
for any $m$ and sufficiently large $n$. Now define the continuous functions, 
\begin{equation}
\varphi^1_m=
\left \{ 
\begin{array}{c}
1 \quad \mbox{on} \,\,\,\,M\backslash \mathbb D_{2^{-m}}(\tilde x)\\
0 \quad \mbox{on}\,\,\,\, \mathbb D_{4^{-m}}(\tilde x)\\
\mbox{harmonic on }\,\,\,\,\mathbb D_{2{-m}}(\tilde x)\backslash \mathbb D_{4^{-m}}(\tilde x)\
\end{array} \right .
\end{equation}
Set $\varphi^2_m=1-\varphi^1_m. $
By the very previous definitions, one then gets 
$$
\|\nabla (\varphi^1_m u)\|_{L^2(M)} +\|\nabla (\varphi^2_m u)\|_{L^2(M)}\leq C (1-\epsilon)^m. 
$$
Consequently, at least one of the two following inequalities holds 
$$
\frac{\|\nabla (\varphi^1_m u)\|^2_{L^2(M)}}{\|(\varphi^1_m u)\|^2_{L^2(M)}} \leq \frac{\|\nabla  u\|^2_{L^2(M)}}{\| u\|^2_{L^2(M)}} +C(1-\epsilon)^m
$$
or 
$$
\frac{\|\nabla (\varphi^2_m u)\|^2_{L^2(M)}}{\|(\varphi^2_m u)\|^2_{L^2(M)}} \leq \frac{\|\nabla  u\|^2_{L^2(M)}}{\| u\|^2_{L^2(M)}} +C(1-\epsilon)^m.
$$

Hence the $N^{th}$ eigenvalue of the disjoint union of $(M, h_{n_m})$ and $(\mathbb S^2, \rho_{n_m})$ is uniformly bounded and hence taking a subsequence we may assume that there exists two numbers $N_1$ and $N_2$ such that $N_1+N_2=N$ and the limits in the lemma exist. 

We now prove the last part of the lemma. Let $b_1 \leq .... \leq b_N$ be the ordering of the set $\left \{ \alpha_0,...,\alpha_{N_1} \right \} \cup \left \{ \beta_0,...,\beta_{N_2} \right \}$ with respect to the multiplicities. From the previous inequalities on the Rayleigh quotients, one has
$$
b_i \leq \lambda_i\quad \,\,\,\,i=1,...,N
$$
For the reverse inequality, let $f$ be a smooth nonnegative function on $M$ bounded by one on $\mathbb D^2$. Let $\lambda_i^0$ be the sequence of eigenvalues of $(m,fg)$ and $\lambda_i^{\epsilon}$ be the eigenvalues on $(M\backslash  \mathbb D_{\epsilon}(\tilde x),fg)$ with zero Boundary data on $\partial \mathbb D_{\epsilon}(\tilde x)$. Therefore, by standard elliptic estimates, one has 
$$
\lim_{\epsilon \to 0} \lambda_i^{\epsilon}=\lambda_i^0. 
$$
Moreover, for any $N,\delta >0$ there exists $\epsilon$ such that 
$$
\lambda_i^{\epsilon} < \lambda_i^0 +\delta,\,\,\,\,\,i=1,...N. 
$$
It follows that 
$$
b_i \geq \lambda_i
$$
and this concludes the lemma. 
\end{proof}

Let $\mu$ be the extremal measure on $\mathbb S^2$ defined in Theorem \ref{main1} for the third eigenvalue $\lambda_3$. Composing metrics $g_n$ with suitable M\"obius  transformations of the sphere we may always assume that the regular part $\mu_r$ of the extremal metric is non identically zero. 
Indeed, in that case we have $K=1$ or $2$. In formulas \eqref{weightSing} and \eqref{weightReg}, we now identify the values of $m_j$ and $m_0$. Since we assume bubbling, the splitting Lemma (see Lemma \ref{decomp}) imposes that $m_0 $ and $m_j$ are different from $3$ and then $m_0,m_j \in \left \{ 1,2 \right \}$. Furthermore, one has 
 $$
\tilde \Lambda_1(\mathbb S^2) = 8\pi
$$
and 
$$
\tilde \Lambda_2(\mathbb S^2) = 16\pi
$$

Therefore, the quotient $c_j/A_r$ belongs to $\left \{ 1,\frac12, 2\right \}$. Owing to the fact that $c_j+A_r=1$ since the total area of the manifold is $1$, this leads to the three follwoing cases 
\begin{enumerate}
\item $A_r=\frac12$
\item $A_r=\frac13$
\item $A_r=\frac23$
\end{enumerate}
By  Theorem  \ref{main1}, one has  
$$\lim_{i\to \infty} \lambda_1(g'_i)=0. $$ Hence by the previous lemma, 
in all cases, one has 

$$\lim_{i\to \infty} \lambda_3(g'_i) \leq \sup \{
\lambda_1( \mu_r g_{round}), \lambda_1(\mu_s g_{round}), \lambda_2( \mu_r g_{round}), \lambda_2(\mu_s g_{round})\}. $$ In case $(3)$, one has 
$$\lim_{i\to \infty} \lambda_3(g'_i) \leq
\lambda_1( 3\mu_s g_{round})= \lambda_2(\frac32\mu_r g_{round}).$$  Hence by Hersch's theorem and by  \cite{nadirSphere} in cases $(2)$ and $(3)$ the 
 Theorem \ref{lambda3} follows. The case $(1)$ is ruled out  since we
are not obtaining an extremal value of the third eigenvalue.

\subsection{No bubbling phenomenon and contradiction}

We now assume that there is no bubbling, i.e. $\mu_s \equiv 0$ (see Theorem \ref{main1}) and we want to reach a contradiction, i.e. there is always bubbling and then the previous argument gives the desired result. 

The strategy of the proof is based on the characterization of the extremal metric in terms of minimal immersions into Euclidean spheres. Due to previous results, only two possible immersions occur and we will rule them out separately. 

By a result in \cite{HON}, the multiplicity of $\lambda_3$ on  $\mathbb S^2$ is less than or equal to $5$. Let $\phi$ be the minimal immersion into a standard sphere constructed in Theorem \ref{main1}. Then by a result of Barbosa (and Calabi) (see \cite{barbosa,calabi}), we have that $\ell-1$ is an even integer and since, by construction, $\ell$ is less or equal the multiplicity of the eigenvalue, one has only two cases: either $\ell-1=2$ or $\ell-1=4$. Therefore, only remains two types of minimal immersions denoted 
$$
\phi_1: \mathbb S^2 \to \mathbb S^2
$$
and 
$$
\phi_2: \mathbb S^2 \to \mathbb S^4
$$
given by eigenvectors of the extremal metric. We prove the following theorem. 

\begin{theorem}\label{step1}
Let 
\begin{equation}\label{immersion}
\psi: \mathbb S^2 \to \mathbb S^2
\end{equation}
be a branched conformal immersion into a standard sphere with at least $3$ sheets. Let $f$ be a metric induced on 
$
\mathbb S^2
$
by the immersion $\psi$. Then there are at least $3$ eigenvalues below $2$ on $(\mathbb S^2,f)$. 
\end{theorem}

\begin{remark}
Note that Theorem \ref{lambda3} holds if the number of sheets of the previous minimal immersion $\psi$ is no more than $3$ since $\Lambda_3(\mathbb S^2) \geq 24\pi$. The equality in the last inequality is attained when $(\mathbb S^2,g)$ is bubbling into three equal round spheres.
\end{remark}

Let $\psi_j: (\mathbb S^2,f_j) \to (\mathbb S^2_{round}),\, j=1,2,\dots,$ be a sequence of isometric
immersions and denote $ \{ \lambda_i^j  \}_i$ the sequence of eigenvalues on $(\mathbb S^2,f_j) $. If $\psi_j$ converges uniformly to $\psi$ then $\lambda_i^j$ are convergent to $\lambda_i$ as $j \to \infty$. Thus if
 $x_1,\dots ,x_k\in \mathbb S^2$ are the branching points of the immersion \eqref{immersion}, then for proving  Theorem \ref{step1} we may assume without loss of generality that the ramification numbers of the branching points $(x_i)_{i=1,...,k}$ are equal to $2$, all points $\psi(x_i)$ are different and moreover  any three points
of the set $\{ \psi(x_1),\dots,\psi(x_k)\}$ are not on the same big circle of $\mathbb S^2$.

Let $\gamma \subset \mathbb S^2$ be a simple closed loop. We call $\gamma$ an {\it arc } if there are two branching  points  $x_n,x_m\in \gamma$ and $\psi (\gamma)$ is an arc of a big circle on $\mathbb S^2$ connecting points $\psi(x_n)$ and $\psi(x_m)$. We call also $\gamma$ an arc between 
$x_n$ and $x_m$.

The proof of the previous Theorem \ref{step1} is a consequence of  

\begin{lemma}\label{lem1}
Under the assumptions of Theorem \ref{step1}, there are two nonintersecting simply connected 
domains $\Omega_1,\Omega_2\subset  \mathbb S^2$ such that $\partial \Omega_1$ and $\partial \Omega_2$ are arcs.
\end{lemma}

\begin{proof}
Let $\gamma$ be an  arc between $x_0$ qnd $x_n$ and $\gamma'$ be an  arc between $x_0$ and $x_m$. We claim that $x_n=x_m$.

Assume not. Let $G$ be a simply connected neighborhood of $\psi(\gamma \cup \gamma')$, such that
$\psi^{-1}(G)$ contains no other branching points except of $x_0,x_n,x_m$. Let $D$ be a connected component of 
$\psi^{-1}(G)$ containing the arcs $\gamma$ and $\gamma'$. Since the branching index of $x_0$ is $1$ the map $\psi: D \to \mathbb \psi(D)\subset \mathbb S^2$ is a two-sheet branching immersion. Since there are three branching
points on $D$ of branching indexes $1$ $\partial D$ is a closed simple loop and hence $D$ is a topological disk. On the other hand by Riemann-Hurwitz theorem the map $\psi: D \to \mathbb S^2$ has exactly one branching point. This is a contradiction.

Let $\gamma, \gamma'$ be two arcs between two different pairs of branching points $x_0,x_n$ and
$y_0,y_n$. We claim that $\gamma \cap \gamma' =\emptyset$. 

Assume not. Let $G$ be a simply connected neighborhood of $\psi(\gamma \cup \gamma')$, such that
$\psi^{-1}(G)$ contains no other branching points except of $x_0,x_n,y_0,y_n$. Let $D$ be a connected component of 
$\psi^{-1}(G)$ containing the arcs $\gamma$ and $\gamma'$. Since $\gamma$ and $\gamma'$ are simple closed curves on $\mathbb S^2$ they have exactly two points of intersection which are projected into a one point by map $\psi$. Hence the map $\psi: D \to \mathbb \psi(D)\subset \mathbb S^2$ is a two-sheet branching immersion. Since there are four  branching
points on $D$ of branching indexes $1$ $\partial D$ is a union of two disjoint closed simple loops and hence $D$ is a topological annulus. On the other hand again by Riemann-Hurwitz theorem the map $\psi: D \to \mathbb S^2$ has exactly two branching points, hence a contradiction.

For $i=1,...,k-1$ we denote by $\ell_i$ a segment of the big circle on $\mathbb S^2$ connecting $\psi (x_1)$ to $\psi (x_{i+1})$. We will show that there exists an index $j$ such that $1 \leq j \leq k-1$ and 
$$
D=\mathbb S^2 \backslash \psi^{-1}(\ell_j)
$$
is disconnected. To prove that we may assume without loss of generality that any three branching points $x_1,x_2,x_3$ (say) are not on the same big circle otherwise we can shift each of them without changing the topological structure of the covering. Denote now
$$
L=\bigcup_{i=1}^{k-1} \ell_i,
$$
$$
\Gamma=\psi^{-1}(L)
$$
and 
$$
G=\mathbb S^2 \backslash \Gamma. 
$$

The set $\Gamma$ is a union of simple arcs on $\mathbb S^2$ with their intersections in the set 
$\psi^{-1}(x_1). 
$
Assume that $\Gamma$ does not separate $\mathbb S^2$. Then the monodromy transformation of $\mathbb S^2$ corresponding to a closed loop surrounding $x_1$ is trivial, or in other words $x_1$ is not a branching point. Consequently, the set $G$ is disconnected. Therefore there exists $j$ as before such that $D$ is disconnected. 

Thus there exists an arc between $x_1$ and $x_j$. Therefore for any $n,\, 1\leq n\leq k$ there exist a
unique $m,\, 1\leq m\leq k$ such that there is an arc between $x_n$ and $x_m$. Denote by
$\gamma_1,\dots,\gamma_l, \, 2l=k$, arcs between  the corresponding pairs of branching points.
As we proved before, the arcs $\gamma_i$ have no mutual intersections. Denote 
$$\Gamma = \bigcup_{1\leq i
\leq l} \gamma_i$$
Then $ \Gamma$ separate $\mathbb S^2$ on $l+1$ domains among which there are at least
two simply connected, which we can take as $\Omega_1, \Omega_2$. This gives the desired result.

\end{proof}

\subsection*{Proof of Theorem \ref{step1}} Let $\Omega_1, \Omega_2$ be the domains defined in Lemma
\ref{lem1}. Then $\partial \Omega_1, \partial \Omega_2$ are arcs and $s_1:=\psi (\partial \Omega_1),
s_2:=\psi (\partial \Omega_2)$ are subsets of big circles, say of $S_1, S_2\subset  \mathbb S^2$.

The maps
$$
\psi: \Omega_i \to \mathbb S^2\setminus s_i,
$$
$i=1,2$, are diffeomorphisms and $\psi^{-1}(S_1)$ ( resp. $\psi^{-1}(S_2)$) separates $\Omega_1$
 (resp. $\Omega_2$) into two domains $D_1$ and $D_2$. Thus we have four domains $D_1,D_2,D_3,D_4$
 on $ \mathbb S^2, \, D_1,\dots, D_4$ without mutual intersection, such that $\psi (D_i)$ are hemispheres. Therefore the ground state of $(D_i,f) $ with a Dirichlet boundary condition on 
 $\partial D_i$ is equal to $2$. Hence the third eigenvalue of $\mathbb S^2$ is $\leq 2$ and since
 $D_1\cup D_2\cup D_3\cup D_4$ is a proper subset of $\mathbb S^2$ the last inequality
 is strict: $\lambda_3<2$.

\qed

\vspace{1cm}
After ruling out the case of minimal immersion from $\mathbb S^2$ into $\mathbb S^2$ , we now rule out the case of the other minimal immersions.

\begin{lemma}\label{step2}
There is no map $\psi$ which is a minimizing harmonic map from $\mathbb S^2$ with the extremal metric into the Euclidean sphere $\mathbb S^4$. 
\end{lemma}
\begin{proof}
Assume by contradiction that such a map exists, i.e. the map $\phi_2$ introduced before is nontrivial. The map $\phi_2$ corresponds to $\ell=5$, i.e. the maximal possible multiplicity for $\lambda_3$ by \cite{HON}. Then we can write (by definition of the discrete spectrum counted with multiplicity) 
$$
\lambda_3=...=\lambda_7. 
$$
Assuming that $\mathbb S^2$ endorsed by the metric induced by the immersion we will have $2=\lambda_3=...=\lambda_7$.
By a result of Calabi \cite{calabi}, we know also that if $\psi$ is a minimal immersion then the area is quantized
$$
A_g(\psi(\mathbb S^2))=4\pi\,d
$$
where $d$ is an integer. Assume that 
$$
A_g(\phi_2(\mathbb S^2))=4\pi\,n
$$

We introduce another minimal immersion $ \phi_3$ generated by eigenvectors from $\mathbb S^2$ into $ \mathbb S^2 \subset \mathbb S^4$ such that $\phi_3:\mathbb S^2\to \mathbb S^2$ is a $n$-sheeted branched covering. By a deep result of Loo \cite{loo}, see also \cite{Kotani} the set $\mathcal M$ of {\sl all } minimal immersions from $\mathbb S^2$ into $\mathbb S^4$ is connected for fixed area, which is the case by a result of Calabi. Indeed,  it is locally even a manifold (see \cite{Verdier}) so it is path-connected. Therefore, there exists a continuous deformation 
of minimal immersions from $\mathbb S^2$ into $\mathbb S^4$ connecting  the minimal immersions $\psi=\phi_2$ and $\phi_3$. By Theorem \ref{step1} the eigenvalue equal to $2$ of $(\mathbb S^2, \tilde g)$ is at least the fourth eigenvalue and $\tilde g$ is the pull-back metric by $\phi_3$. As previously mentioned, $\psi$ is a minimal immersion associated to an harmonic map given by eigenvectors with an extremal eigenvalue of multiplicity $5$. However, since for $\phi_2$ the eigenvalue equal to $2$ is the third one and since the spectrum is continuously depends on the deformation this implies that there is a minimal immersion $\tilde \psi$ along the deformation such that for $\tilde \psi$ the eigenvalue $2$ is still
the third eigenvalue for the corresponding metric on $\mathbb S^2$ and  there is at least one additional eigenvalue coming into the value $2$ and raising  the multiplicity of $2$ at least to $6$ and then by Lemma \ref{step1}, we reach a contradiction.
\end{proof}

\bibliographystyle{alpha}
\bibliography{biblio}

\medskip

{\em NN} -- 
CNRS, I2M UMR 7353-- Centre de Math\'ematiques et Informatique, Marseille, France. 
 
{\tt nikolay.nadirashvili@univ-amu.fr}

\medskip

{\em YS} --  
Johns Hopkins University, Krieger Hall, Baltimore, USA. 

{\tt sire@math.jhu.edu} 

\end{document}